\documentclass{article}
\usepackage{amsbsy}
\usepackage{amsmath}
\usepackage{amssymb}
\usepackage{amsthm}
\usepackage{amsmath}
\usepackage{graphicx}
\usepackage{float}
\usepackage{cite}
\usepackage[latin1]{inputenc}

\theoremstyle{definition}  \newtheorem{Theorem}{Theorem}
\theoremstyle{definition}  \newtheorem{Corollary}{Corollary}
\theoremstyle{definition}  \newtheorem{Lemma}{Lemma}
\theoremstyle{definition}  \newtheorem{Example}{Example}
\theoremstyle{definition}  \newtheorem{Remark}{Remark}
 \numberwithin{Theorem}{section}  \numberwithin{Remark}{section}
\numberwithin{Lemma}{section} \numberwithin{Corollary}{section}
\numberwithin{Example}{section} \numberwithin{equation}{section}

\title{\Large \bf Second-order differential equation with indefinite and repulsive singularities
\footnote{ Yonghui Xia acknowledges the support from 
National Natural
Science Foundation of China under Grant (No. 11671176, 11931016).}
}
\author
{
Xiaoxiao Cui$^{a}$\,\,\,\,\,
Yonghui Xia$^{a}\footnote{Corresponding author. Yonghui Xia, xiadoc@163.com;yhxia@zjnu.cn.
}$
\\
{\small \textit{$^a$ College of Mathematics and Computer Science,  Zhejiang Normal University, 321004, Jinhua, China}}\\
{\small Email:xxc\_2017@126.com; xiadoc@163.com; yhxia@zjnu.cn.}
}

\date{}
\begin{document}
\maketitle
\begin{center}
\begin{minipage}{13cm}
 {\bf Abstract}---This paper concerns a second-order differential equation with indefinite and repulsive singularities. It is the first time to study differential equation containing both indefinite and repulsive singularities simultaneously. A set of sufficient conditions are obtained for the existence of positive periodic solutions.  The theoretical underpinnings of this paper are the
positivity of Green's function and fixed point theorem in cones. Our results improve and extend the results in previous literatures. Finally,  three examples and their numerical simulations (phase diagrams and time diagrams of periodic solutions) are given to show the effectiveness of our conclusions.  
\end{minipage}
\end{center}

\vspace{1cm}
\section{\Large{Introduction}}

\par\noindent

There are many works on the problem of periodic solutions for the nonlinear ordinary differential equations with deviating argument, $p$-Laplacian, neutral type operator and singularity (see e.g. \cite{Amster2007, Bereanu2009, Bereanu2007, Boscaggin2020-1, Boscaggin2020, Bravo2010, Cabada1997, CX2021, Cheng2021, Cheng2018, Cheung2006, Chu2012, ChuNA, Chu2009, ChuBLMS2008, ChuBLMS2007, Chu-JDE, Fabry1986, Fonda2017, Godoy2021, Godoy2019, Hakl2010, Hakl2011, Hakl2017, Han2009, Jebelean2002, ChuD2, Jiang2005, Kong2017, Kong2015, Lazer1987, Li2017, Li2008, Lu2004, Lu2019-1, Lu2019-2, Lu2019, Martins2006, Omari1984, Regan1997, Peng2007, Pino1993, Pino1992, Torres2015,TorresNA,TorresJDE, Urena2016, Wang2007, Wang2014, Wang2010, Yu2019, Yu2022}). In this paper, we pay particular attention to the differential equations with singularity. In fact, the differential equations with singularity has attracted many mathematicians' attention (see e.g.\cite{Bereanu2007, Bravo2010, Cheng2021, CX2021, Cheng2018, Chu2012, ChuNA, Chu2009, ChuBLMS2008, ChuBLMS2007, Chu-JDE, Fonda2017, Godoy2021, Godoy2019, Hakl2010, Hakl2011, Hakl2017, Jebelean2002, ChuD2, Jiang2005, Kong2017, Kong2015, Lazer1987, Li2017, Li2008, Lu2019-1, Lu2019-2, Lu2019, Martins2006, Pino1992, Torres2015,TorresNA,TorresJDE, Urena2016, Wang2014, Yu2019, Yu2022}). By the method of upper and lower solution, Lazer and Solimini \cite{Lazer1987} initially studied the periodic solutions of the following singular equations
\begin{equation}\label{eq1.4}
    x''(t)+\frac{1}{x^\mu(t)}=m(t),
\end{equation}
and
\begin{equation}\label{eq1.5}
    x''(t)-\frac{1}{x^\mu(t)}=m(t),
\end{equation}
 where $m(t)\in C(\mathbb{R},\mathbb{R})$ is a periodic function and $\mu>0$ is a constant. It is said that equation \eqref{eq1.4} has an attractive singularity and equation \eqref{eq1.5} has a repulsive singularity.
Pino and Man\'{a}sevich  \cite{Pino1992} studied the periodic solutions for a problem arising in nonlinear
elasticity. The equations is formulated as
\begin{equation}\label{eq1.6aa}
    x''(t)+f(t,x(t))=0.
\end{equation} Further, Pino, Man\'{a}sevich and Montero,\cite{Pino1992, Pino1993} considered the equation \eqref{eq1.6aa} with singularities.
 Hakl and Torres \cite{Hakl2010} discussed  equation \eqref{eq1.6aa} with  attractive-repulsive singularities:
\begin{equation}\label{eq1.6}
    x''(t)=\frac{b(t)}{x^{\nu_1}(t)}-\frac{c(t)}{x^{\nu_2}(t)}+e(t),
\end{equation}
where $b(t),~c(t)\in C(\mathbb{R},(0,+\infty))$ are periodic functions, $\nu_1>\nu_2>0$ are constants. The positivity of $b(t),~c(t)$ states the equation has both attractive and repulsive singularities. Bravo and Torres \cite{Bravo2010}   investigated equation \eqref{eq1.6aa} with an indefinite singularity as follows:
\begin{equation}\label{eq1.6a}
    x''(t)=\frac{b(t)}{x^3(t)},
\end{equation}
where $b(t)$ is a piecewise-constant sign-changing function. Motivated by the two works \cite{Bravo2010,Hakl2010}, the mathematicians paid their attention to the singular equations containing  both attractive and repulsive singularities simultaneously or indefinite singularity (see \cite{Cheng2021, Cheng2018, Godoy2019, Hakl2017, Lu2019-1, Lu2019-2, Lu2019, Urena2016, Yu2022}). Thereinto, Hakl and Zamora \cite{Hakl2017}, Godoy and Zamora \cite{Godoy2019} generalized equation \eqref{eq1.6a} in a wider application.  In the recent years, Chu et. al. \cite{ChuBLMS2007, Chu-JDE, Jiang2005}, Li and Zhang \cite{Li2008}, Torres \cite{TorresNA,TorresJDE} obtained some existence results for positive periodic solutions of the following equation by some fixed point theorems
\begin{equation}{\label{eq1.7aa}}
 x''(t)+a(t)x(t)=f(t,x)+e(t),
 \end{equation}
where $f(t,x)$ may be singular at $x=0$. Subsequently, Cheng and Cui \cite{Cheng2021} considered equation \eqref{eq1.7aa} with $f(t,x)$ an indefinite singularity, i.e.,
\begin{equation}{\label{eq1.8}}
 x''(t)+q(t)x(t)=\frac{b(t)}{x^\rho(t)}+e(t),
 \end{equation}
where $\rho>0$ is a constant and $b(t)\in C(\mathbb{R},\mathbb{R})$ is a periodic function.%
 Lu et. al. \cite{Lu2019-1, Lu2019-2, Lu2019} by means of coincidence degree theory generalized indefinite singularity to Li\'enard equation.

In this paper, we study the following differential equation with indefinite and repulsive singularities simultaneously:
\begin{equation}\label{eq1.1}
    x''(t)+p(t)x'(t)+q(t)x(t)=\frac{b(t)}{x^{\rho_1}(t)}+\frac{c(t)}{x^{\rho_2}(t)}+e(t),
\end{equation}
where $p(t),~q(t),~b(t)\in C(\mathbb{R},\mathbb{R})$ and $c(t),~e(t)\in C(\mathbb{R},(0,+\infty))$ are $\omega$-periodic functions, $\rho_1,~\rho_2$ are positive constants. The term $\frac{b(t)}{x^{\rho_1}(t)}$ has an indefinite singularity in view of the uncertainty of the sign of weight function $b(t)$, even the singularity will be disappeared if $b(t)=0$ in some subintervals. In addition to this, the equation has a repulsive singularity because of the term $\frac{c(t)}{x^{\rho_2}(t)}$ and $c(t)>0$. In particular, if equation \eqref{eq1.1} has no damping term (i.e., $p(t)\equiv0$), it reduces to the following differential equation:
\begin{equation}\label{eq1.3}
    x''(t)+q(t)x(t)=\frac{b(t)}{x^{\rho_1}(t)}+\frac{c(t)}{x^{\rho_2}(t)}+e(t).
\end{equation}
We generalize and improve the previous works in three aspects. {\em Firstly}, equation \eqref{eq1.1} is investigated in a very general form containing both indefinite singularity and repulsive singularity simultaneously. Obviously, it is   more general and complex than equations \eqref{eq1.6aa}-\eqref{eq1.8}. In fact, equations \eqref{eq1.6aa}-\eqref{eq1.8} either do not contain indefinite singularity or single singularity. This means that the previous methods 
can not be directly applied to equation \eqref{eq1.1} any more. {\em Secondly}, noticing the term $\frac{b_1(t)}{x^{\nu_1}}-\frac{c_1(t)}{x^{\nu_2}}$, Chu et al. \cite{Chu2012}, Cheng and Ren \cite{Cheng2018}, Hakl and Torres \cite{Hakl2010} required $\nu_1>\nu_2$. In fact, $\nu_1>\nu_2$ implies that equation \eqref{eq1.6} tends to be more repulsive singularity. In this paper, we consider all the three cases, $\rho_1>\rho_2$, $\rho_1=\rho_2$ and $\rho_1<\rho_2$. Consequently, our results generalize and improve  the existing literatures.
{\em Thirdly}, our results can be applied to both strong singularity and weak singularity.
In fact,  following Lazer and Solimini's work, there are many good papers appearing in this field by different kinds of methods  such as  the fixed point theorems \cite{Cheng2021, CX2021, ChuD2}, lower and upper solution \cite{Amster2007, Hakl2010, Hakl2011}, Leray-Schauder alternative principle \cite{Chu2012, Li2017, Jiang2005} and coincidence degree theory \cite{Boscaggin2020, Lu2019, Kong2017}, and so on. Different from above methods, in this paper, we use the theory of the fixed point in cones to study the periodic solution.

 The rest structure of this paper is as follows. In section 2, the sufficient conditions for the positivity of Green's function of the second-order differential equation are given.~
 Our main results are stated in Section 3. 
 In section 4, three examples and their simulations are given to confirm our conclusions.

\section{\Large{Positivity of Green's function}}

\par\noindent

In order to make good use of fixed point theorem to get the existence of positive periodic solution for equation \eqref{eq1.2}, first of all we need to guarantee the invariance of the sign of Green's function of the nonhomogeneous linear equation corresponding to equation \eqref{eq1.2}. According to the specific situation of this paper, we consider the positivity of Green's function.

Let $C_\omega^1:=\{\phi(t),~\phi'(t)\in C(\mathbb{R},\mathbb{R}):~\phi(t+\omega)=\phi(t),~\phi'(t+\omega)=\phi'(t),~t\in\mathbb{R}\}$, equipped with the norm $\|\phi\|=\|\phi\|_{\infty}+\|\phi'\|_{\infty}$, and $C_\omega^1$ in the norm $\|\cdot\|$ sense is a Banach space, here $\|\phi\|_{\infty}=\max\limits_{t\in[0,\omega]}|\phi(t)|$. For a continuous function $\phi(t)$, we define $\phi^*:=\max\limits_{0\leq t\leq\omega}\phi(t)$ and $\phi_*:=\min\limits_{0\leq t\leq\omega}\phi(t)$.

Considering the following second-order nonhomogeneous linear differential equation,
\begin{equation}\label{eq2.1}
\begin{cases}
u''(t)+p(t)u'(t)+l(t)u(t)=h(t),\\
u(0)=u(\omega),~u'(0)=u'(\omega),
\end{cases}
\end{equation}
where $l(t)=\frac{q(t)}{\alpha}$, $h(t)\in C(\mathbb{R},(0,+\infty))$ is a continuous $\omega$-periodic function. The unique $\omega$-periodic solution of equation \eqref{eq2.1} is expressed as
$$u(t)=\int^\omega_0 G(t,s)h(s)ds,$$
where $G(t,s)$ is Green's function of equation \eqref{eq2.1}. Now we talk about sufficient conditions to make $G(t,s)>0$ such that $\omega$-periodic solution of equation \eqref{eq2.1} is positive. To describe the problem conveniently, we assume the following condition holds:

$(A)$ The Green's function $G(t,s)>0$ for all $(t,s)\in[0,\omega]\times[0,\omega]$.

In 2005, Wang et al.\cite[Lemma 2.4]{Wang2007} proved condition $(A)$ can be satisfied under the following two conditions:

$(A_1)$ There are continuous $\omega$-periodic functions $a_1(t)$ and
$a_2(t)$ such that $\int^\omega_0a_1(t)dt>0,~~\int^\omega_0a_2(t)dt>0$
and
$$a_1(t)+a_2(t)=p(t),~~~~a_1'(t)+a_1(t)a_2(t)=l(t),~~~~\mbox{for}~~t\in\mathbb{R}.$$

$(A_2)$ $\left(\int^\omega_0p(t)dt\right)^2\geq
4\omega^2\exp\left(\frac{1}{\omega}\int^\omega_0 \ln
l(s)ds\right).$\\

On the basis of Wang's work, Cheng and Ren gave the following conclusion:

\begin{Lemma} {\label{Theorem 2.1}}\cite[Lemma 2.2]{Cheng2018} Assume that condition $(A_1)$ holds, then condition $(A)$ holds.
\end{Lemma}

In 2012, Chu et al. \cite{Chu2012} also proved condition $(A)$ by using antimaximum principle. First, authors defined two functions
$$\varsigma(p)(t)=exp\left(\int^t_0p(s)ds\right)$$and$$\varsigma_1(p)(t)=\varsigma(p)(\omega)\int^t_0p(s)ds+\int^\omega_t\varsigma(p)(s)ds,$$
and then gave the following theorem:

\begin{Lemma}{\label{Theorem 2.2}}\cite[Corollary 2.6]{Chu2012} Assume that
$$\int^\omega_0l(s)\varsigma(p)(s)\varsigma_1(-p)(s)ds\geq0$$
and
$$\sup\limits_{0\leq t\leq\omega}\left\{\int^{t+\omega}_t\varsigma(-p)(s)ds\int^{t+\omega}_t\max\{l(s),0\}\varsigma(p)(s)ds\right\}\leq4$$
hold. If $l(t)\not\equiv0$,
then condition $(A)$ holds.
\end{Lemma}

In particular, as the damping term of equation \eqref{eq2.1} is absent (i.e., $p(t)\equiv0$), a specific value about Green's function can be obtained, and condition $(A)$ is ensured.

\begin{Lemma}{\label{Theorem 2.3}}\cite[Lemma 2.5]{Han2009} In the case $l(t)=\xi^2$ with $\xi>0$, the expression of the Green's function as follows,
\begin{equation}\label{eq2.2}
G(t,s)=\left\{
\begin{aligned}
&\frac{\cos\xi(t-s-\frac{\omega}{2})}{2\xi\sin\frac{\xi\omega}{2}},~~~~ 0\leq s\leq t\leq\omega,\\
&\frac{\cos\xi(t-s+\frac{\omega}{2})}{2\xi\sin\frac{\xi\omega}{2}},~~~~
0\leq t<s\leq\omega.
 \end{aligned}
 \right.
\end{equation}
If
\begin{equation}\label{eq2.3}
\xi<\frac{\pi}{\omega},
\end{equation}
then condition $(A)$ holds. 
\end{Lemma}

Under the circumstances that condition $(A)$ holds, one can say $$G^*:=\max\limits_{0\leq s,t\leq\omega}G(t,s)>G_*:=\min\limits_{0\leq s,t\leq\omega}G(t,s)>0,$$
and we define $$\sigma:=\frac{G_*}{G^*+\max\limits_{0\leq
s,t\leq\omega}\left|\frac{\partial G(t,s)}{\partial t}\right|},~~\delta:=\frac{\max\limits_{0\leq
s,t\leq\omega}\left|\frac{\partial G(t,s)}{\partial t}\right|}{G_*}.$$

\section{\Large{Existence of positive periodic solution for equation \eqref{eq1.1}}}

\par\noindent

Our main results and their proofs are presented in this section. %
To study eqution \eqref{eq1.1}, we make a transformation of variable $x=y^\alpha$ and $\alpha=\frac{1}{\rho_1+1}$. Then equation \eqref{eq1.1} is rewritten as follows:
\begin{equation}\label{eq1.2}
    y''(t)+p(t)y'(t)+\frac{q(t)}{\alpha}y(t)=\frac{c(t)}{\alpha}y^{1-\alpha-\alpha\rho_2}(t)+\frac{e(t)}{\alpha}y^{1-\alpha}(t)+(1-\alpha)\frac{|y'(t)|^2}{y(t)}+\frac{b(t)}{\alpha}.
\end{equation}
Obviously, by this transformation, it no longer contains indefinite singularity. We see that the term $(1-\alpha)\frac{|y'(t)|^2}{y(t)}$ has a repulsive singularity. Now let's consider the term $\frac{c(t)}{\alpha}y^{1-\alpha-\alpha\rho_2}(t)$:

$(i)$ If $\rho_1>\rho_2$ (i.e. $0<1-\alpha-\alpha\rho_2<1$), the term has no singularity, in this case, equation \eqref{eq1.2} has just a repulsive singularity;

$(ii)$ if $\rho_1<\rho_2$ (i.e. $1-\alpha-\alpha\rho_2<0$), the term has a repulsive singularity, in this case, equation \eqref{eq1.2} has two repulsive singularities;

$(iii)$ if $\rho_1=\rho_2$ (i.e. $1-\alpha-\alpha\rho_2=0$), same situation as $(i)$.

Consequently, the existence of positive periodic solution for equation \eqref{eq1.1} is equivalent to that of equation \eqref{eq1.2}.  To prove the existence of positive periodic solutions, we need a fixed
point theorem in cones (see \cite{Regan1997}).

\begin{Lemma}{\label{Lemma 3.1}} Let $X$ be a Banach space and $K$ is a
cone in $X$. Assume that $\Lambda_1,~\Lambda_2$ are open subsets of
$X$ with $0\in\Lambda_1,~\overline{\Lambda}_1\subset\Lambda_2$. Let
$$\mathcal{T}: K\cap(\overline{\Lambda}_2\backslash\Lambda_1)\rightarrow K$$
be a continuous and completely continuous operator such that

$(i)$ $\|\mathcal{T}y\|\leq\|y\|$ for $y\in K\cap\partial\Lambda_2$;

$(ii)$ there exists $y_0\in K\backslash\{0\}$ such that $y\neq
\mathcal{T}y+\lambda y_0$ for $y\in K\cap\partial\Lambda_1$ and $\lambda>0$.\\
Then $\mathcal{T}$ has a fixed point in
$K\cap(\overline{\Lambda}_2\backslash\Lambda_1)$.
\end{Lemma}

All the discussion throughout this paper takes place in Banach space $C^1_\omega$, define a cone $K$ in space $C^1_\omega$ as follows:
$$K=\{y\in C^1_\omega:~\min y(t)\geq\sigma\|y\|,~|y'(t)|\leq\delta y(t),~0\leq t\leq\omega\}.$$

In what follows, we discuss the periodic solutions with three cases: $\rho_1>\rho_2$, $\rho_1<\rho_2$ and $\rho_1=\rho_2$.

\begin{Theorem} {\label{Theorem 3.1}} Assume that condition $(A)$ and $\rho_1>\rho_2$ hold.
Furthermore, suppose there exist $R>r>0$ such that:

$(H_1)$ $\frac{1}{\sigma}\left(-\frac{b_*}{e_*}\right)^{\frac{1}{1-\alpha}}\leq r\leq\left(\frac{G_*c_*\omega\sigma^{1-\alpha-\alpha\rho_2}}{\alpha}\right)^{\frac{1}{\alpha+\alpha\rho_2}}$.

$(H_2)$ $\frac{(1-\alpha)(G^*+\delta G_*)\delta^2\omega}{\sigma}<1$.\\
Then there is at least one
positive $\omega$-periodic solution of equation \eqref{eq1.2}, saying $y(t)$ satisfying $r\leq\|y\|\leq R$. Consequently, it is a positive periodic solution of equation \eqref{eq1.1}.
\end{Theorem}

\begin{proof} To use the fixed point in cones, take $$\Lambda_1:=\{y\in C_\omega^1: \|y\|<r\},~~~\Lambda_2:=\{y\in C_\omega^1: \|y\|<R\},$$
and define an operator $\mathcal{T}$ as
$$(\mathcal{T}y)(t)=\int^{\omega}_0G(t,s)\left(\frac{c(s)}{\alpha}y^{1-\alpha-\alpha\rho_2}(s)+\frac{e(s)}{\alpha}y^{1-\alpha}(s)+(1-\alpha)\frac{|y'(s)|^2}{y(s)}+\frac{b(s)}{\alpha}\right)ds,$$
then a fixed point of operator equation $y=\mathcal{T}y$ is an $\omega$-periodic solution of equation \eqref{eq1.2}.
 To this end, we divide it into four steps.

\noindent {\bf Step 1.} We prove $\mathcal{T}(K\cap(\overline{\Lambda}_2\backslash\Lambda_1))\subset K$.
In fact, for
$\forall~y\in K\cap(\overline{\Lambda}_2\backslash\Lambda_1)$, by definitions of $K,~\Lambda_1$ and $\Lambda_2$, one can say $$\sigma r\leq\sigma\|y\|\leq y(t)\leq\|y\|_{\infty}\leq\|y\|\leq R~\mbox{and}~|y'(t)|\leq\delta y(t)\leq\delta R.$$ It follows from the first half of condition $(H_1)$ that
\begin{equation}\label{eq3.1}
\begin{split}
&\frac{c(t)}{\alpha}y^{1-\alpha-\alpha\rho_2}(t)+\frac{e(t)}{\alpha}y^{1-\alpha}(t)+(1-\alpha)\frac{|y'(t)|^2}{y(t)}+\frac{b(t)}{\alpha}\\
\geq&\frac{c(t)}{\alpha}y^{1-\alpha-\alpha\rho_2}(t)+\frac{e(t)}{\alpha}y^{1-\alpha}(t)+\frac{b(t)}{\alpha}\\
\geq&\frac{c_*\sigma^{1-\alpha-\alpha\rho_2}}{\alpha}r^{1-\alpha-\alpha\rho_2}+\frac{e_*\sigma^{1-\alpha}}{\alpha}r^{1-\alpha}+\frac{b_*}{\alpha}\\
>&\frac{e_*\sigma^{1-\alpha}}{\alpha}r^{1-\alpha}+\frac{b_*}{\alpha}\\
\geq&0.
\end{split}
\end{equation}
On the basis of condition $(A)$ and \eqref{eq3.1}, we have the following two inequalities,
\begin{equation*}
\begin{split}
    &\min\limits_{0\leq t\leq\omega}(\mathcal{T}y)(t)\\
    =&\min\limits_{0\leq t\leq\omega}\int^{\omega}_0G(t,s)\left(\frac{c(s)}{\alpha}y^{1-\alpha-\alpha\rho_2}(s)+\frac{e(s)}{\alpha}y^{1-\alpha}(s)+(1-\alpha)\frac{|y'(s)|^2}{y(s)}+\frac{b(s)}{\alpha}\right)ds\\
    \geq&G_*\int^{\omega}_0\left(\frac{c(s)}{\alpha}y^{1-\alpha-\alpha\rho_2}(s)+\frac{e(s)}{\alpha}y^{1-\alpha}(s)+(1-\alpha)\frac{|y'(s)|^2}{y(s)}+\frac{b(s)}{\alpha}\right)ds\\
    =&\sigma\left(G^*+\max\limits_{0\leq s,t\leq\omega}\left|\frac{\partial G(t,s)}{\partial t}\right|\right) \int^{\omega}_0\left(\frac{c(s)}{\alpha}y^{1-\alpha-\alpha\rho_2}(s)+\frac{e(s)}{\alpha}y^{1-\alpha}(s)+(1-\alpha)\frac{|y'(s)|^2}{y(s)}+\frac{b(s)}{\alpha}\right)ds\\
    =&\sigma G^* \int^{\omega}_0\left(\frac{c(s)}{\alpha}y^{1-\alpha-\alpha\rho_2}(s)+\frac{e(s)}{\alpha}y^{1-\alpha}(s)+(1-\alpha)\frac{|y'(s)|^2}{y(s)}+\frac{b(s)}{\alpha}\right)ds\\
    &+\sigma\max\limits_{0\leq s,t\leq\omega}\left|\frac{\partial G(t,s)}{\partial t}\right| \int^{\omega}_0\left(\frac{c(s)}{\alpha}y^{1-\alpha-\alpha\rho_2}(s)+\frac{e(s)}{\alpha}y^{1-\alpha}(s)+(1-\alpha)\frac{|y'(s)|^2}{y(s)}+\frac{b(s)}{\alpha}\right)ds\\
    \geq&\sigma\max\limits_{0\leq t\leq\omega}\int^{\omega}_0G(t,s)\left(\frac{c(s)}{\alpha}y^{1-\alpha-\alpha\rho_2}(s)+\frac{e(s)}{\alpha}y^{1-\alpha}(s)+(1-\alpha)\frac{|y'(s)|^2}{y(s)}+\frac{b(s)}{\alpha}\right)ds\\
    &+\sigma\max\limits_{0\leq t\leq\omega}\int^{\omega}_0\left|\frac{\partial G(t,s)}{\partial t}\right|\left(\frac{c(s)}{\alpha}y^{1-\alpha-\alpha\rho_2}(s)+\frac{e(s)}{\alpha}y^{1-\alpha}(s)+(1-\alpha)\frac{|y'(s)|^2}{y(s)}+\frac{b(s)}{\alpha}\right)ds\\
    =&\sigma\|\mathcal{T}y\|_{\infty}+\sigma\|(\mathcal{T}y)'\|_{\infty}=\sigma\|\mathcal{T}y\|,
\end{split}
\end{equation*}
and
\begin{equation*}
\begin{split}
    |(\mathcal{T}y)'(t)|=&\left|\int^{\omega}_0\frac{\partial G(t,s)}{\partial t}\left(\frac{c(s)}{\alpha}y^{1-\alpha-\alpha\rho_2}(s)+\frac{e(s)}{\alpha}y^{1-\alpha}(s)+(1-\alpha)\frac{|y'(s)|^2}{y(s)}+\frac{b(s)}{\alpha}\right)ds\right|\\
    \leq&\int^{\omega}_0\left|\frac{\partial G(t,s)}{\partial t}\right|\left(\frac{c(s)}{\alpha}y^{1-\alpha-\alpha\rho_2}(s)+\frac{e(s)}{\alpha}y^{1-\alpha}(s)+(1-\alpha)\frac{|y'(s)|^2}{y(s)}+\frac{b(s)}{\alpha}\right)ds\\
    \leq&\max\limits_{0\leq t\leq\omega}\left|\frac{\partial G(t,s)}{\partial t}\right|\int^{\omega}_0\left(\frac{c(s)}{\alpha}y^{1-\alpha-\alpha\rho_2}(s)+\frac{e(s)}{\alpha}y^{1-\alpha}(s)+(1-\alpha)\frac{|y'(s)|^2}{y(s)}+\frac{b(s)}{\alpha}\right)ds\\
    =&\delta G_*\int^{\omega}_0\left(\frac{c(s)}{\alpha}y^{1-\alpha-\alpha\rho_2}(s)+\frac{e(s)}{\alpha}y^{1-\alpha}(s)+(1-\alpha)\frac{|y'(s)|^2}{y(s)}+\frac{b(s)}{\alpha}\right)ds\\
    \leq&\delta\int^{\omega}_0G(t,s)\left(\frac{c(s)}{\alpha}y^{1-\alpha-\alpha\rho_2}(s)+\frac{e(s)}{\alpha}y^{1-\alpha}(s)+(1-\alpha)\frac{|y'(s)|^2}{y(s)}+\frac{b(s)}{\alpha}\right)ds\\
    =&\delta(\mathcal{T}y)(t).
\end{split}
\end{equation*}
This shows that $\mathcal{T}(K\cap(\overline{\Lambda}_2\backslash\Lambda_1))\subset K$.

\noindent {\bf Step 2.}  We  prove that $\mathcal{T}$ is a completely continuous operator. For any $y\in
K\cap(\overline{\Lambda}_2\backslash\Lambda_1)$, we have
\begin{equation*}\label{eq3.2}
\begin{split}
    |\mathcal{T}y|=&\left|\int^{\omega}_0G(t,s)\left(\frac{c(s)}{\alpha}y^{1-\alpha-\alpha\rho_2}(s)+\frac{e(s)}{\alpha}y^{1-\alpha}(s)+(1-\alpha)\frac{|y'(s)|^2}{y(s)}+\frac{b(s)}{\alpha}\right)ds\right|\\
    =&\int^{\omega}_0G(t,s)\left(\frac{c(s)}{\alpha}y^{1-\alpha-\alpha\rho_2}(s)+\frac{e(s)}{\alpha}y^{1-\alpha}(s)+(1-\alpha)\frac{|y'(s)|^2}{y(s)}+\frac{b(s)}{\alpha}\right)ds\\
    \leq&G^*\left(\frac{c^*}{\alpha}R^{1-\alpha-\alpha\rho_2}+\frac{e^*}{\alpha}R^{1-\alpha}+(1-\alpha)\frac{(\delta R)^2}{\sigma r}+\frac{b^*}{\alpha}\right)\omega:=N_1,
\end{split}
\end{equation*}
and
\begin{equation*}\label{eq3.2*}
\begin{split}
    |(\mathcal{T}y)'(t)|=&\left|\int^{\omega}_0\frac{\partial G(t,s)}{\partial t}\left(\frac{c(s)}{\alpha}y^{1-\alpha-\alpha\rho_2}(s)+\frac{e(s)}{\alpha}y^{1-\alpha}(s)+(1-\alpha)\frac{|y'(s)|^2}{y(s)}+\frac{b(s)}{\alpha}\right)ds\right|\\
    \leq&\delta G_*\int^{\omega}_0\left(\frac{c(s)}{\alpha}y^{1-\alpha-\alpha\rho_2}(s)+\frac{e(s)}{\alpha}y^{1-\alpha}(s)+(1-\alpha)\frac{|y'(s)|^2}{y(s)}+\frac{b(s)}{\alpha}\right)ds\\
    \leq&\delta G_*\left(\frac{c^*}{\alpha}R^{1-\alpha-\alpha\rho_2}+\frac{e^*}{\alpha}R^{1-\alpha}+(1-\alpha)\frac{(\delta R)^2}{\sigma r}+\frac{b^*}{\alpha}\right)\omega:=N_2.
\end{split}
\end{equation*}
Furthermore, for any $ y_1,~y_2\in
K\cap(\overline{\Lambda}_2\backslash\Lambda_1)$ with $y_1\neq y_2$,
\begin{equation*}
|\mathcal{T}y_1-\mathcal{T}y_2|=\left|\frac{d\mathcal{T}(\theta y_1+(1-\theta)y_2)}{dt}(y_1-y_2)\right|\leq N_2|y_1-y_2|,
\end{equation*}
where $0<\theta<1$. That is to say, $\mathcal{T}$ is a uniformly bounded and equicontinuity operator, it follows the Arzela-Ascoli theorem that $\mathcal{T}: K\cap(\overline{\Lambda}_2\backslash\Lambda_1)\rightarrow K$ is a continuous and completely continuous operator.

\noindent {\bf Step 3.} We prove $(i)$ of Lemma \ref{Lemma 3.1}, i.e.,
\begin{equation}\label{eq3.3}
    \|\mathcal{T}y\|\leq\|y\|,~~\mbox{for}~~y\in K\cap\partial\Lambda_2.
\end{equation}
In view of $y\in K\cap\partial\Lambda_2$, there is
$$\sigma R\leq y(t)\leq\|y\|_{\infty}\leq\|y\|=R~\mbox{and}~|y'|\leq\delta y(t)\leq\delta R.$$
Therefore, we arrive at
\begin{equation}\label{eq3.4}
    \begin{split}
    \|\mathcal{T}y\|=&\|\mathcal{T}y\|_{\infty}+\|(\mathcal{T}y)'\|_{\infty}\\
    =&\max\limits_{0\leq t\leq\omega}\int^{\omega}_0G(t,s)\left(\frac{c(s)}{\alpha}y^{1-\alpha-\alpha\rho_2}(s)+\frac{e(s)}{\alpha}y^{1-\alpha}(s)+(1-\alpha)\frac{|y'(s)|^2}{y(s)}+\frac{b(s)}{\alpha}\right)ds\\
    &+\max\limits_{0\leq t\leq\omega}\int^{\omega}_0\left|\frac{\partial G(t,s)}{\partial t}\right|\left(\frac{c(s)}{\alpha}y^{1-\alpha-\alpha\rho_2}(s)+\frac{e(s)}{\alpha}y^{1-\alpha}(s)+(1-\alpha)\frac{|y'(s)|^2}{y(s)}+\frac{b(s)}{\alpha}\right)ds\\
    \leq&(G^*+\delta G_*)\left(\frac{c^*}{\alpha}R^{1-\alpha-\alpha\rho_2}+\frac{e^*}{\alpha}R^{1-\alpha}+\frac{(1-\alpha)\delta^2}{\sigma}R+\frac{b^*}{\alpha}\right)\omega.
    \end{split}
\end{equation}
In view of $1-\alpha-\alpha\rho_2<1$ and $1-\alpha<1$, condition $(H_2)$ implies that there is a sufficiently large $R$ such that
\begin{equation*}
(G^*+\delta G_*)\left(\frac{c^*}{\alpha}R^{1-\alpha-\alpha\rho_2}+\frac{e^*}{\alpha}R^{1-\alpha}+\frac{(1-\alpha)\delta^2}{\sigma}R+\frac{b^*}{\alpha}\right)\omega\leq R=\|y\|.
\end{equation*}
Therefore, inequality \eqref{eq3.3} holds.

\noindent {\bf Step 4.} We prove $(ii)$ of Lemma \ref{Lemma 3.1}. To do so, let $y_0=1$, then $y_0\in K\backslash\{0\}$. Now we prove that
\begin{equation}\label{eq3.5}
    y\neq \mathcal{T}y+\lambda y_0,~~\mbox{for}~~y\in
K\cap\partial\Lambda_1,~~\lambda>0
\end{equation}
by way of contradiction. If inequality \eqref{eq3.5} does not hold, we can assume that there exist $y_1\in K\cap\partial\Lambda_1$
and $\lambda_0>0$ such that
$$y_1=\mathcal{T}y_1+\lambda_0 y_0.$$ To be sure, for $y_1\in
K\cap\partial\Lambda_1$, it holds that
$$\sigma r=\sigma\|y_1\|\leq y_1(t)\leq\|y_1\|_{\infty}\leq\|y_1\|=r~\mbox{and}~|y'_1(t)|\leq\delta y_1(t)\leq\delta r.$$ Therefore, from condition $(H_1)$, we obtain
\begin{equation}\label{eq3.6}
    \begin{split}
    y_1=&\mathcal{T}y_1+\lambda_0 y_0\\
    =&\int^{\omega}_0G(t,s)\left(\frac{c(s)}{\alpha}y_1^{1-\alpha-\alpha\rho_2}(s)+\frac{e(s)}{\alpha}y_1^{1-\alpha}(s)+(1-\alpha)\frac{|y'_1(s)|^2}{y_1(s)}+\frac{b(s)}{\alpha}\right)ds+\lambda_0\\
    >&\int^{\omega}_0G(t,s)\left(\frac{c(s)}{\alpha}y_1^{1-\alpha-\alpha\rho_2}(s)+\frac{e(s)}{\alpha}y_1^{1-\alpha}(s)+\frac{b(s)}{\alpha}\right)ds\\
    \geq&G_*\left(\frac{c_*}{\alpha}(\sigma r)^{1-\alpha-\alpha\rho_2}+\frac{e_*}{\alpha}(\sigma r)^{1-\alpha}+\frac{b_*}{\alpha}\right)\omega\\
    \geq&\frac{G_*c_*\sigma^{1-\alpha-\alpha\rho_2}\omega}{\alpha}r^{1-\alpha-\alpha\rho_2}\\
    \geq&r=\|y_1\|,
    \end{split}
\end{equation}
which contradicts with $y_1\in K\cap\partial\Lambda_1$. Hence, inequality \eqref{eq3.5} holds.

It follows from above four steps that all the conditions of Lemma \ref{Lemma 3.1} are satisfied. Therefore, it follows that $\mathcal{T}$ has a fixed point $y\in K\cap(\overline{\Lambda}_2\backslash\Lambda_1)$.
Obviously, the fixed point is an $\omega$-periodic solution of
equation \eqref{eq1.2}. That is, equation \eqref{eq1.2} has at least one positive $\omega$-periodic solution $y(t)$ satisfying $r\leq\|y\|\leq R$.
\end{proof}

\begin{Corollary}{\label{Corollary 3.1}} Assume that inequality \eqref{eq2.3}, conditions $(H_1),~(H_2)$ and $\rho_1>\rho_2$ hold. Then equation \eqref{eq1.3} has at least one positive $\omega$-periodic solution.
\end{Corollary}

\begin{Theorem}{\label{Theorem 3.2}} Assume that conditions $(A),~(H_2)$ and $\rho_1<\rho_2$ hold.
Furthermore, suppose there exist $R>r>0$ such that:

$(H_3)$ $\frac{1}{\sigma}\left(-\frac{b_*}{e_*}\right)^{\frac{1}{1-\alpha}}\leq r\leq\left(\frac{G_*c_*\omega}{\alpha}\right)^{\frac{1}{\alpha+\alpha\rho_2}}$.\\
Then there is at least one
positive $\omega$-periodic solution of equation \eqref{eq1.2}, saying $y(t)$ satisfying $r\leq\|y\|\leq R$. Consequently, it is a positive periodic solution of equation \eqref{eq1.1}.
\end{Theorem}

\begin{proof} We follow the same strategy and notations as in the proof of Theorem \ref{Theorem 3.1}. Now we prove \eqref{eq3.3}, similar to \eqref{eq3.4}, by using condition $(H_2)$ and $\rho_1<\rho_2$, one can prove that
\begin{equation*}
    \begin{split}
    \|\mathcal{T}y\|=&\|\mathcal{T}y\|_{\infty}+\|(\mathcal{T}y)'\|_{\infty}\\
    =&\max\limits_{0\leq t\leq\omega}\int^{\omega}_0G(t,s)\left(\frac{c(s)}{\alpha}y^{1-\alpha-\alpha\rho_2}(s)+\frac{e(s)}{\alpha}y^{1-\alpha}(s)+(1-\alpha)\frac{|y'(s)|^2}{y(s)}+\frac{b(s)}{\alpha}\right)ds\\
    &+\max\limits_{0\leq t\leq\omega}\int^{\omega}_0\left|\frac{\partial G(t,s)}{\partial t}\right|\left(\frac{c(s)}{\alpha}y^{1-\alpha-\alpha\rho_2}(s)+\frac{e(s)}{\alpha}y^{1-\alpha}(s)+(1-\alpha)\frac{|y'(s)|^2}{y(s)}+\frac{b(s)}{\alpha}\right)ds\\
    \leq&(G^*+\delta G_*)\left(\frac{c^*\sigma^{1-\alpha-\alpha\rho_2}}{\alpha}R^{1-\alpha-\alpha\rho_2}+\frac{e^*}{\alpha}R^{1-\alpha}+\frac{(1-\alpha)\delta^2}{\sigma}R+\frac{b^*}{\alpha}\right)\omega\\
    \leq&R=\|y\|
    \end{split}
\end{equation*}
holds if $R$ is large enough. This means that inequality \eqref{eq3.3} holds.

In what follows, we prove that inequality \eqref{eq3.5}, from condition $(H_3)$ and $\rho_1<\rho_2$, \eqref{eq3.6} gives
\begin{equation*}
    \begin{split}
    y_1=&\mathcal{T}y_1+\lambda_0 y_0\\
    =&\int^{\omega}_0G(t,s)\left(\frac{c(s)}{\alpha}y_1^{1-\alpha-\alpha\rho_2}(s)+\frac{e(s)}{\alpha}y_1^{1-\alpha}(s)+(1-\alpha)\frac{|y'_1(s)|^2}{y_1(s)}+\frac{b(s)}{\alpha}\right)ds+\lambda_0\\
    >&G_*\left(\frac{c_*}{\alpha}r^{1-\alpha-\alpha\rho_2}+\frac{e_*}{\alpha}(\sigma r)^{1-\alpha}+\frac{b_*}{\alpha}\right)\omega\\
    \geq&\frac{G_*c_*\omega}{\alpha}r^{1-\alpha-\alpha\rho_2}\\
    \geq&r=\|y_1\|,
    \end{split}
\end{equation*}
which contradicts with $y_1\in K\cap\partial\Lambda_1$. Then inequality \eqref{eq3.5} follows.
Therefore, the conclusion follows from Lemma \ref{Lemma 3.1} immediately. 
\end{proof}

\begin{Corollary}{\label{Corollary 3.2}} Assume that inequality \eqref{eq2.3}, conditions $(H_2),~(H_3)$ and $\rho_1<\rho_2$ hold. Then there is at least one positive $\omega$-periodic solution of equation \eqref{eq1.3}.
\end{Corollary}

\begin{Theorem} {\label{Theorem 3.3}} Assume that conditions $(A),~(H_2)$ and $\rho_1=\rho_2$ hold.
Furthermore, suppose there exist $R>r>0$ such that one of the following two cases holds:

Case I: $b_*+c_*\geq0$;

Case II: $b_*+c_*<0$ and condition $(H_4)$ hold:

$(H_4)$ $\frac{1}{\sigma}\left(-\frac{b_*+c_*}{e_*}\right)^{\frac{1}{1-\alpha}}<r\leq\frac{G_*\bar{b}^+\omega}{\alpha}$,
where $b^+(t):=\max\{0,b(t)\},~\bar{b}^+:=\frac{1}{\omega}\int_0^\omega b^+(t)dt$.\\
Then there is at least one
positive $\omega$-periodic solution of equation \eqref{eq1.2}, saying $y(t)$ satisfying $r\leq\|y\|\leq R$. Consequently, it is a positive periodic solution of equation \eqref{eq1.1}.
\end{Theorem}

\begin{proof} If $\rho_1=\rho_2$, equation \eqref{eq1.2} can be rewritten as
\begin{equation*}
    y''(t)+p(t)y'(t)+\frac{q(t)}{\alpha}y(t)=\frac{e(t)}{\alpha}y^{1-\alpha}(t)+(1-\alpha)\frac{|y'(t)|^2}{y(t)}+\frac{b(t)+c(t)}{\alpha},
\end{equation*}
and operator $\mathcal{T}$ as
$$(\mathcal{T}y)(t)=\int^{\omega}_0G(t,s)\left(\frac{e(s)}{\alpha}y^{1-\alpha}(s)+(1-\alpha)\frac{|y'(s)|^2}{y(s)}+\frac{b(s)+c(s)}{\alpha}\right)ds.$$

Similar to the proof of Theorem \ref{Theorem 3.1}. For $\forall~y\in K\cap(\overline{\Lambda}_2\backslash\Lambda_1)$, we have
\begin{equation*}
\begin{split}
&\frac{e(t)}{\alpha}y^{1-\alpha}(t)+(1-\alpha)\frac{|y'(t)|^2}{y(t)}+\frac{b(t)+c(t)}{\alpha}\\
\geq&\frac{e(t)}{\alpha}y^{1-\alpha}(t)+\frac{b(t)+c(t)}{\alpha}\\
\geq&\frac{e_*\sigma^{1-\alpha}}{\alpha}r^{1-\alpha}+\frac{b_*+c_*}{\alpha}\\
>&0.
\end{split}
\end{equation*}

Case I: if $b_*+c_*\geq0$, it is obvious that the above inequality holds;

Case II: if $b_*+c_*<0$, from the first half part of condition $(H_4)$, we can get the above inequality holds as well.

Furthermore, one can prove that $\mathcal{T}(K\cap(\overline{\Lambda}_2\backslash\Lambda_1))\subset K$ and $\mathcal{T}: K\cap(\overline{\Lambda}_2\backslash\Lambda_1)\rightarrow K$ is a continuous and completely continuous operator.

Now we prove inequality \eqref{eq3.3}, similar to \eqref{eq3.4}, by using condition $(H_2)$ and $\rho_1=\rho_2$, one can prove that
\begin{equation*}
    \begin{split}
    \|\mathcal{T}y\|=&\|\mathcal{T}y\|_{\infty}+\|(\mathcal{T}y)'\|_{\infty}\\
    =&\max\limits_{0\leq t\leq\omega}\int^{\omega}_0G(t,s)\left(\frac{c(s)}{\alpha}y^{1-\alpha-\alpha\rho_2}(s)+\frac{e(s)}{\alpha}y^{1-\alpha}(s)+(1-\alpha)\frac{|y'(s)|^2}{y(s)}+\frac{b(s)}{\alpha}\right)ds\\
    &+\max\limits_{0\leq t\leq\omega}\int^{\omega}_0\left|\frac{\partial G(t,s)}{\partial t}\right|\left(\frac{c(s)}{\alpha}y^{1-\alpha-\alpha\rho_2}(s)+\frac{e(s)}{\alpha}y^{1-\alpha}(s)+(1-\alpha)\frac{|y'(s)|^2}{y(s)}+\frac{b(s)}{\alpha}\right)ds\\
    \leq&(G^*+\delta G_*)\left(\frac{e^*}{\alpha}R^{1-\alpha}+\frac{(1-\alpha)\delta^2}{\sigma}R+\frac{b^*+c^*}{\alpha}\right)\omega\\
    \leq&R=\|y\|
    \end{split}
\end{equation*}
holds if $R$ is large enough. This means that inequality \eqref{eq3.3} holds.

In the following, we prove that inequality \eqref{eq3.5}.

Case I: if $b_*+c_*\geq0$, the inequality
\begin{equation}\label{eq3.7}
    \begin{split}
    y_1=&\mathcal{T}y_1+\lambda_0 y_0\\
    =&\int^{\omega}_0G(t,s)\left(\frac{e(s)}{\alpha}y_1^{1-\alpha}(s)+(1-\alpha)\frac{|y'_1(s)|^2}{y_1(s)}+\frac{b(s)+c(s)}{\alpha}\right)ds+\lambda_0\\
    >&G_*\left(\frac{e_*}{\alpha}(\sigma r)^{1-\alpha}+\frac{b_*+c_*}{\alpha}\right)\omega\\
    \geq&\frac{G_*e_*\sigma^{1-\alpha}\omega}{\alpha}r^{1-\alpha}\\
    \geq&r=\|y_1\|
    \end{split}
\end{equation}
can be satisfied if $r>0$ is small enough.

Case II: if $b_*+c_*<0$, from condition $(H_4)$, we have
\begin{equation}\label{eq3.8}
    \begin{split}
    y_1=&\int^{\omega}_0G(t,s)\left(\frac{e(s)}{\alpha}y_1^{1-\alpha}(s)+(1-\alpha)\frac{|y'_1(s)|^2}{y_1(s)}+\frac{b(s)+c(s)}{\alpha}\right)ds+\lambda_0\\
    >&G_*\int^{\omega}_0\left(\frac{e_*}{\alpha}(\sigma r)^{1-\alpha}+\frac{b^+(s)}{\alpha}+\frac{b^-(s)}{\alpha}+\frac{c_*}{\alpha}\right)ds\\
    \geq&G_*\int^{\omega}_0\left(\frac{e_*\sigma^{1-\alpha}}{\alpha}r^{1-\alpha}+\frac{b^+(s)}{\alpha}+\frac{b_*}{\alpha}+\frac{c_*}{\alpha}\right)ds\\
    >&G_*\int^{\omega}_0\frac{b^+(s)}{\alpha}ds\\
    =&\frac{G_*\bar{b}^+\omega}{\alpha}\\
    \geq&r=\|y_1\|,
    \end{split}
\end{equation}
where $b^-(t):=\min\{0,b(t)\}$. Both cases I and II contradict with $y_1\in K\cap\partial\Lambda_1$, so inequality \eqref{eq3.5} follows.

Therefore, Lemma \ref{Lemma 3.1} shows equation \eqref{eq1.2} has a positive $\omega$-periodic solution $y(t)$ satisfying $r\leq\|y\|\leq R$.
\end{proof}

\begin{Corollary}{\label{Corollary 3.3}} Assume that equation \eqref{eq2.3}, condition $(H_2)$ and $\rho_1=\rho_2$ hold. Then equation \eqref{eq1.3} has at least one positive $\omega$-periodic solution $y(t)$ with $r\leq\|y\|\leq R$ if one of the following two cases holds:

Case I: $b_*+c_*\geq0$;

Case II: $b_*+c_*<0$ and condition $(H_4)$ hold.
\end{Corollary}

\begin{Remark}
 Corollary \ref{Corollary 3.3} generalizes the results in literature \cite{Cheng2021}.
\end{Remark}

\section{\Large{Example}}

\par\noindent

In this section, we give three examples and the simulations to illustrate our results.

\begin{Example}{\label{Example 4.1}}
Consider differential equation with indefinite and repulsive singularities:
\begin{equation}{\label{eq4.1}}
x''+\frac{1}{40}x=\frac{1+2\cos(3t)}{x^{\frac{3}{2}}}+\frac{e^{2\sin(3t)}}{x^{\frac{13}{10}}}+10+\cos(3t).
\end{equation}

Comparing equation \eqref{eq4.1} with equation \eqref{eq1.3}, we see that $q(t)=\frac{1}{40}$, $b(t)=1+2\cos(3t)$, $b_*=-1$, $c(t)=e^{2\sin(3t)}$, $c_*=e^{-2}$, $e(t)=10+\cos(3t)$, $e_*=9$, $\rho_1=\frac{3}{2}>\rho_2=\frac{13}{10}$, so we can apply Corollary \ref{Corollary 3.1}.

In view of $\alpha=\frac{1}{1+\rho_1}=\frac{2}{5}$, $\frac{q(t)}{\alpha}=\frac{1}{16}$, $\xi=\sqrt{\frac{q(t)}{\alpha}}=\frac{1}{4}$, $\omega=\frac{2\pi}{3}$, we know $\xi<\frac{\pi}{\omega}$, equation \eqref{eq2.3} hold. From equation \eqref{eq2.2}, we can obtain $G_*=\frac{2\sqrt{2+\sqrt{3}}}{\sqrt{2-\sqrt{3}}}$, $G^*=\frac{4}{\sqrt{2-\sqrt{3}}}$, $\max\limits_{0\leq s,t\leq\omega}\left|\frac{\partial G(t,s)}{\partial t}\right|=\frac{\sqrt{2-\sqrt{3}}}{2}$, $\sigma=\frac{G_*}{G^*+\max\limits_{0\leq s,t\leq\omega}\left|\frac{\partial G(t,s)}{\partial t}\right|}=\frac{4\sqrt{2+\sqrt{3}}}{10-\sqrt{3}}$, $\delta=\frac{\max\limits_{0\leq
s,t\leq\omega}\left|\frac{\partial G(t,s)}{\partial t}\right|}{G_*}=\frac{2-\sqrt{3}}{4\sqrt{2+\sqrt{3}}}$.

Conditions $(H_1)$ and $(H_2)$ can be satisfied by calculation. So by using Corollary \ref{Corollary 3.1}, equation \eqref{eq4.1} has at least one positive $\frac{2\pi}{3}$-periodic solution.

\begin{figure}[H]
\centering
   \includegraphics[width=1\textwidth]{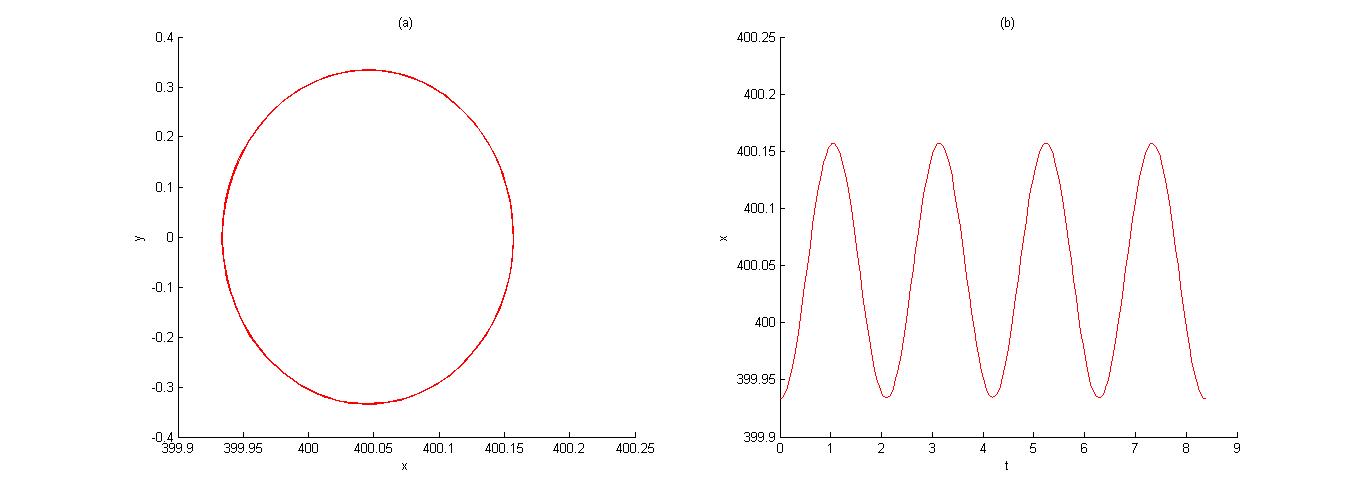}
\caption{Phase diagram of the $\frac{2\pi}{3}$-periodic solution and
its time series diagram. (a) Phase diagram of the periodic solution with initial value
$(399.93015,0)$. (b) Time series diagram of the periodic solution.}
\end{figure}

\end{Example}

\begin{Example}{\label{Example 4.2}}
Consider differential equation with indefinite and repulsive singularities:
\begin{equation}{\label{eq4.2}}
x''+\frac{1}{40}x=\frac{1+2\cos(3t)}{x^{\frac{3}{2}}}+\frac{e^{2\sin(3t)}}{x^2}+10+\cos(3t).
\end{equation}

Since $\rho_1=\frac{3}{2}<\rho_2=2$, we can apply Corollary \ref{Corollary 3.2}. Similar to Example \ref{Example 4.1}, we can get equation \eqref{eq2.3}, conditions $(H_2)$ and $(H_3)$ hold. So by using Corollary \ref{Corollary 3.2}, equation \eqref{eq4.2} has at least one positive $\frac{2\pi}{3}$-periodic solution.

\begin{figure}[H]
\centering
   \includegraphics[width=1\textwidth]{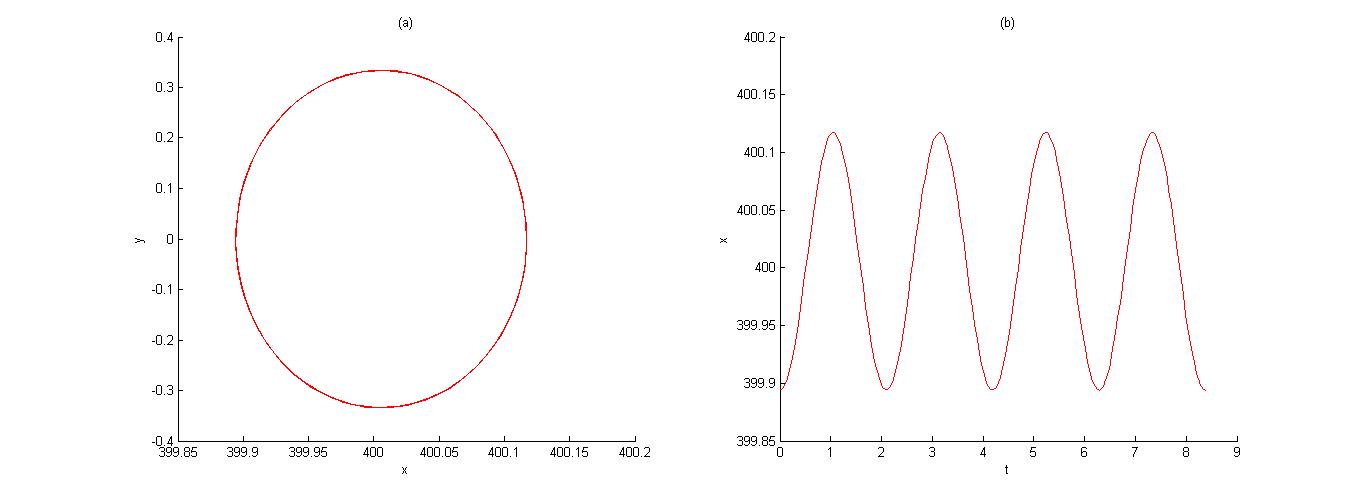}
\caption{Phase diagram of the $\frac{2\pi}{3}$-periodic solution and
its time series diagram. (a) Phase diagram of the periodic solution with initial value
$(399.8941,0)$. (b) Time series diagram of the periodic solution.}
\end{figure}

\end{Example}

\begin{Example}{\label{Example 4.3}}
Consider differential equation with indefinite and repulsive singularities:
\begin{equation}{\label{eq4.3}}
x''+\frac{1}{40}x=\frac{1+2\cos(3t)}{x^{\frac{3}{2}}}+\frac{e^{2\sin(3t)}}{x^{\frac{3}{2}}}+10+\cos(3t).
\end{equation}

Obviously, $\rho_1=\frac{3}{2}=\rho_2$ and $b_*+c_*<0$. Moreover,
$$\bar{b}^+=\frac{1}{\omega}\int_0^\omega b^+(t)dt=\frac{3}{2\pi}\int_{-\frac{2\pi}{9}}^{\frac{2\pi}{9}}(1+2\cos(3t))dt=\frac{2}{3}+\frac{\sqrt{3}}{\pi}.$$
By Example \ref{Example 4.1}, we can get equation \eqref{eq2.3}, conditions $(H_2)$ and $(H_4)$ hold. So by using Case II of Corollary \ref{Corollary 3.3}, equation \eqref{eq4.3} has at least one positive $\frac{2\pi}{3}$-periodic solution.

\begin{figure}[H]
\centering
  \includegraphics[width=1\textwidth]{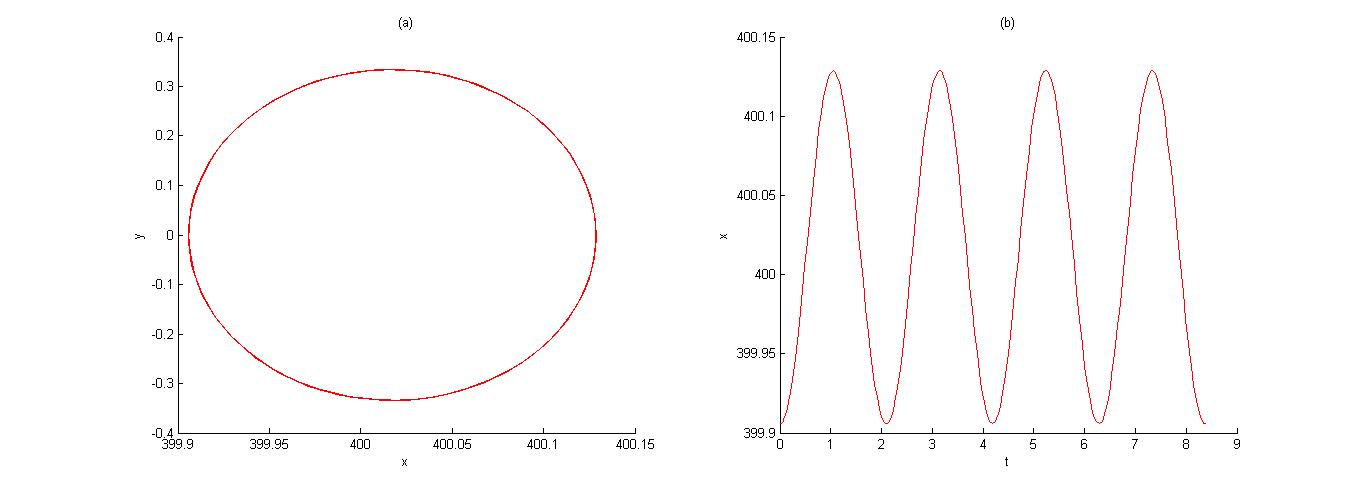}
\caption{Phase diagram of the $\frac{2\pi}{3}$-periodic solution and
its time series diagram. (a) Phase diagram of the periodic solution with initial value
$(399.9045,0)$. (b) Time series diagram of the periodic solution.}
\end{figure}

\end{Example}

\section*{Data Availability Statement}

My manuscript has no associated data. It is pure mathematics.

\section*{Conflict of interest statement}
 The authors declare that there is no conflict of interests regarding the publication of this article.

\section*{Contributions}
 We declare that all the authors have same contributions to this paper.


\begin{thebibliography}{99}


\bibitem{Amster2007}P. Amster, M. Mariani, Oscillating solutions of a nonlinear fourth order ordinary differential equation, J. Math. Anal. Appl., {\bf 325} (2007) 1133-1141.


\bibitem{Bereanu2009}C. Bereanu, Periodic solutions of some fourth-order nonlinear differential equations, Nonlinear Anal., {\bf 71} (2009) 53-57.

\bibitem{Bereanu2007}C. Bereanu, J. Mawhin, Existence and multiplicity results for some nonlinear problems with singular $\varphi$-Laplacian, J. Differential Equations, {\bf 243} (2007) 536-557.

\bibitem{Boscaggin2020-1}A. Boscaggin, G. Feltrin, Pairs of positive radial solutions for a Minkowski-curvature Neumann problem with indefinite weight, Nonlinear Anal., {\bf 196} (2020) 14 pp.


\bibitem{Boscaggin2020}A. Boscaggin, G. Feltrin, Positive periodic solutions to an indefinite Minkowski-curvature equation, J. Differential Equations, {\bf 269} (2020) 5595-5645.

\bibitem{Bravo2010} J. Bravo, P. Torres, Periodic solutions of a singular equation with indefinite weight, Adv. Nonlinear Stud., {\bf 10} (2010) 927-938.

\bibitem{Cabada1997}A. Cabada, S. Lois, Maximum principles for fourth and sixth order periodic boundary value problems, Nonlinear Anal., {\bf 29} (1997) 1161-1171.

\bibitem{Cheng2021} Z. Cheng, X. Cui, Positive periodic solution to an indefinite singular equation, Appl. Math. Lett.,  {\bf 112}  (2021) 106740.

\bibitem{CX2021} Z. Cheng, X. Cui, Positive periodic solution for a second-order damped singular equation via fixed point theorem in cones, Bull. Malays. Math. Sci. Soc., {\bf 44} (2021) 2675-2691.


\bibitem{Cheng2018} Z. Cheng, J. Ren, Periodic solution for second order damped
differential equations with attractive-repulsive singularities,
Rocky Mountain J. Math., {\bf 48} (2018) 753-768.


\bibitem{Cheung2006} W. Cheung, J. Ren, Periodic solutions for $p$-Laplacian Rayleigh equations, Nonlinear Anal., {\bf 65} (2006) 2003-2012.



\bibitem{Chu2012} J. Chu, N. Fan, P. Torres, Periodic solutions for
second order singular damped differential equations, J. Math. Anal.
Appl., {\bf 388} (2012) 665-675.

\bibitem{ChuNA}  J. Chu, M. Li,  Positive periodic solutions of Hill's equations with singular nonlinear perturbations, Nonlinear Anal., {\bf 69} (2008) 276-286.

\bibitem{Chu2009}  J. Chu, M. Li,  Twist periodic solutions of second order singular differential equations, J. Math. Anal. Appl., {\bf 355} (2009) 830-838.



\bibitem{ChuBLMS2008}  J. Chu, J. Nieto, Impulsive periodic solutions of first-order singular differential equations, Bull. Lond. Math. Soc., {\bf 40} (2008) 143-150.

\bibitem{ChuBLMS2007}J. Chu, P. Torres, Applications of Schauder's fixed point theorem to singular differential equations, Bull. Lond. Math. Soc., {\bf 39} (2007) 653-660.

\bibitem{Chu-JDE}J. Chu, P. Torres, M. Zhang,  Periodic solutions of second order non-autonomous singular dynamical systems, J. Differential Equations, {\bf 239} (2007) 196-212.



\bibitem{Fabry1986} C. Fabry, J. Mawhin, M. Nkashma, A multiplicity result for periodic solutions of forced nonlinear second order ordinary differential equations, Bull. Lond. Math. Soc., {\bf 18} (1986) 173-180.

\bibitem{Fonda2017}A. Fonda, A. Sfecci, On a singular periodic Ambrosetti-Prodi problem, Nonlinear Anal., {\bf 149} (2017) 146-155.

\bibitem{Godoy2021}J. Godoy, R. Hakl, X. Yu, Existence and multiplicity of periodic solutions to differential equations with attractive singularities, Proc. R. Soc. Edinb., Sect. A, Math., (2021) 1-26.


\bibitem{Godoy2019} J. Godoy, M. Zamora, Periodic solutions for a second-order differential equation with indefinite weak singularity, Proc. Roy. Soc. Edinburgh Sect. A, {\bf 149} (2019) 1135-1152.


\bibitem{Hakl2010} R. Hakl, P. Torres, On periodic solutions of second-order differential equations with attractive-repulsive singularities, J. Differential Equations,  {\bf 248}  (2010) 111-126.


\bibitem{Hakl2011} R. Hakl, P. Torres, M. Zamora, Periodic solutions of singular second order differential equations: Upper and lower functions, Nonlinear Anal., {\bf 74} (2011) 7078-7093.




\bibitem{Hakl2017} R. Hakl, M. Zamora, Periodic solutions to second-order
indefinite singular equations, J. Differential Equations, {\bf 263} (2017) 451-469.


\bibitem{Han2009} W. Han, J. Ren,  Some results on second-order neutral functional
 differential equations with infinite distributed
delay, Nonlinear Anal., {\bf 70} (2009) 1393-1406.


\bibitem{Jebelean2002} P. Jebelean, J. Mawhin, Periodic solutions of singular nonlinear perturbations of the ordinary $p$-Laplacian, Adv. Nonlinear Stud., {\bf 2} (2002) 299-312.


\bibitem{ChuD2} D. Jiang, J. Chu, D. O'Regan, R. Agarwal, Multiple positive solutions to superlinear periodic boundary value problems with repulsive singular forces, J. Math. Anal. Appl., {\bf 286} (2003) 563-576.


\bibitem{Jiang2005} D. Jiang, J. Chu, M. Zhang, Multiplicity of
positive periodic solutions to superlinear repulsive singular
equations, J. Differential Equations, {\bf 211} (2005) 282-302.


\bibitem{Kong2017} F. Kong, S. Lu, Existence of positive periodic solutions of fourth-order singular $p$-Laplacian neutral functional differential equations, Filomat,  {\bf 31}  (2017) 5855-5868.


\bibitem{Kong2015} F. Kong, S. Lu, Z. Liang, Existence of positive periodic solutions for neutral Li\'enard differential equations with a singularity, Electron. J. Differential Equations, {\bf 242} (2015) 1-12.


\bibitem{Lazer1987} A. Lazer, S. Solimini, On periodic solutions of
nonlinear differential equations with singularities, Proc. Amer.
Math. Soc., {\bf 99} (1987) 109-114.


\bibitem{Li2017} S. Li, Y. Wang, Multiplicity of positive periodic solutions to second order singular dynamical systems, Mediterr. J. Math.,  {\bf 14} (2017) 13 pp.


\bibitem{Li2008} X. Li, Z. Zhang, Periodic solutions for second order differential equations with a singular nonlinearity, Nonlinear Anal., {\bf 69} (2008) 3866-3876.





\bibitem{Lu2004}S. Lu, W. Ge, Z. Zhen, Periodic solutions for a kind of Rayleigh equation with a deviating argument, Appl. Math. Lett., {\bf 17} (2004) 443-449.


\bibitem{Lu2019-1} S. Lu, Y. Guo, L. Chen, Periodic solutions for Li\'enard equation with an indefinite singularity, Nonlinear Anal. Real World Appl., {\bf 45} (2019) 542-556.


\bibitem{Lu2019-2} S. Lu, R. Xue, Periodic solutions for a singular Li\'enard equation with indefinite weight, Topol. Methods Nonlinear Anal., {\bf 54} (2019) 203-218.


\bibitem{Lu2019} S. Lu, X. Yu, Existence of positive periodic solutions for Li\'enard equations with an indefinite singularity of attractive type, Bound. Value Probl., {\bf 101} (2018) 19 pp.


\bibitem{Martins2006} R. Martins, Existence of periodic solutions for second-order differential equations with singularities and the strong force condition, J. Math. Anal. Appl., {\bf 317} (2006) 1-13.


\bibitem{Omari1984}P. Omari, F. Zanolin, On forced nonlinear oscillations in nth order differential systems with geometric conditions, Nonlinear Anal., {\bf 8} (1984) 723-748.


\bibitem{Regan1997} D. O'Regan, Existence Theory for Nonlinear Ordinary
Differential Equations, Kluwer Academic Publishers Group, Dordrecht,
(1997).

\bibitem{Peng2007}S. Peng, S. Zhu, Periodic solutions for $p$-Laplacian Rayleigh equations with a deviating argument, Nonlinear Anal., {\bf 67} (2007) 138-146.


\bibitem{Pino1993} M. Pino, R. Man\'{a}sevich, Infinitely many $T$-periodic solutions for a problem arising in nonlinear elasticity, J. Differential Equations, {\bf 103} (1993) 260-277.


\bibitem{Pino1992} M. Pino, R. Man\'{a}sevich, A. Montero, $T$-periodic solutions for some second order differential equations with singularities, Proc. Roy. Soc. Edinburgh Sect. A, {\bf 120} (1992) 231-243.


\bibitem{Torres2015} P. Torres, Mathematical models with singularities-A zoo
of singular creatures. Atlantis Briefs in Differential Equations,
Atlantis Press, Paris, (2015).



\bibitem{TorresNA} P.  Torres, Bounded solutions in singular equations of repulsive type, Nonlinear Anal., 32 (1998) 117-125.

\bibitem{TorresJDE} P.  Torres, Weak singularities may help periodic solutions to exist, J. Differential Equations, 232 (2007) 277-284.

\bibitem{Urena2016} A. Ure\~{n}a, Periodic solutions of singular equations, Topol. Methods Nonlinear Anal.,
{\bf 47} (2016) 55-72.



\bibitem{Wang2007} Y. Wang, H. Lian, W. Ge,  Periodic solutions for a
second order nonlinear functional differential equation, Appl. Math.
Lett., {\bf 20} (2007) 110-115.


\bibitem{Wang2014} Z. Wang, Periodic solutions of Li\'enard equation with a singularity and a deviating argument, Nonlinear Anal. Real World Appl., {\bf 16} (2014) 227-234.


\bibitem{Wang2010} Z. Wang, L. Qian, S. Lu, On the existence of periodic solutions to a fourth-order $p$-Laplacian differential equation with a deviating argument, Nonlinear Anal. Real World Appl., {\bf 11} (2010) 1660-1669.



\bibitem{Yu2019} X. Yu, S. Lu, A multiplicity result for periodic solutions of Li\'enard equations with an attractive singularity, Appl. Math. Comput., {\bf 346} (2019) 183-192.


\bibitem{Yu2022} X. Yu, S. Lu, F. Kong, Existence and multiplicity of positive periodic solutions to Minkowski-curvature equations without coercivity condition, J. Math. Anal. Appl., {\bf 507} (2022) 15 pp.



\end{thebibliography}
\end{document}